\documentclass{amsart}

\usepackage{latexsym,enumerate}
\usepackage{amsmath,amsthm,amsopn,amstext,amscd,amsfonts,amssymb}
\usepackage[ansinew]{inputenc}
\usepackage{verbatim}
\usepackage{graphicx}
\usepackage{pstricks}
\usepackage{wasysym}
\usepackage[all]{xy}
\usepackage{epsfig} 
\usepackage{graphicx,psfrag} 
\usepackage{subfigure}
\usepackage{pstricks,pst-node}

\usepackage{multicol}
\usepackage{color}
\usepackage{colortbl}

\newcommand{\RR}{\mathbb{R}}

\newtheorem{theorem}{Theorem}[section]
\newtheorem{lemma}[theorem]{Lemma}
\newtheorem{proposition}[theorem]{Proposition}
\newtheorem{corollary}[theorem]{Corollary}

\newtheorem{example}[theorem]{Example}
\newtheorem{remark}[theorem]{Remark}

\newcommand{\spb}[1]{\smallskip}
\newcommand{\mpb}[1]{\medskip}
\newcommand{\bpb}[1]{\bigskip}

\renewcommand{\a}{\alpha}
\renewcommand{\b}{\beta}
\newcommand{\e}{\varepsilon}
\renewcommand{\d}{\delta}
\newcommand{\D}{\Delta}
\newcommand{\g}{\gamma}

\renewcommand{\l}{\lambda}

\renewcommand{\O}{\Omega}
\renewcommand{\o}{\omega}

\begin{document}
\DeclareGraphicsExtensions{.jpg,.pdf,.mps,.png}

\title{A new approximation to the geometric-arithmetic index}

\author[Alvaro Mart\'{\i}nez-P\'erez]{Alvaro Mart\'{\i}nez-P\'erez$^{(1)}$}
\address{ Facultad CC. Sociales de Talavera,
Avda. Real Fábrica de Seda, s/n. 45600 Talavera de la Reina, Toledo, Spain}
\email{alvaro.martinezperez@uclm.es}
\thanks{$^{(1)}$ Supported in part by a grant
from Ministerio de Econom{\'\i}a y Competitividad (MTM 2015-63612P), Spain.}

\author[Jos\'e M. Rodr{\'\i}guez]{Jos\'e M. Rodr{\'\i}guez$^{(2)}$}
\address{Departamento de Matem\'aticas, Universidad Carlos III de Madrid,
Avenida de la Universidad 30, 28911 Legan\'es, Madrid, Spain}
\email{jomaro@math.uc3m.es}
\thanks{$^{(2)}$ Supported in part by two grants
from Ministerio de Econom{\'\i}a y Competitividad (MTM 2016-78227-C2-1-P and MTM 2015-69323-REDT), Spain, and a grant from CONACYT (FOMIX-CONACyT-UAGro 249818), M\'exico.}

\author[Jos\'e M. Sigarreta]{Jos\'e M. Sigarreta$^{(2)}$}
\address{Faculdad de Matem\'aticas, Universidad Aut\'onoma de Guerrero, Carlos E. Adame 5, Col. La Garita, Acapulco, Gro., Mexico}
\email{jsmathguerrero@gmail.com}

\date{\today}

\begin{abstract}
The concept of geometric-arithmetic index was introduced in the
chemical graph theory recently, but it has shown to be useful.
The aim of this paper is to obtain new inequalities involving the geometric-arithmetic index $GA_1$
and characterize graphs extremal with respect to them.
\end{abstract}

\maketitle{}




\section{Introduction}
A single number, representing a chemical structure in graph-theoretical terms via the
molecular graph, is called a topological descriptor and if it in addition correlates with
a molecular property it is called topological index, which is used to understand physicochemical
properties of chemical compounds.
Topological indices are interesting since they capture some of the properties of a molecule in a single number.
Hundreds of topological indices have been introduced and studied, starting with the
seminal work by Wiener in which he used the sum of all shortest-path distances of
a (molecular) graph for modeling physical properties of alkanes (see \cite{Wi}).

Topological indices based on end-vertex degrees of edges have been
used over 40 years. Among them, several indices are recognized to be useful tools in
chemical researches. Probably, the best know such descriptor is the Randi\'c connectivity
index ($R$) \cite{R}. There are more than thousand papers and a couple of books dealing with
this molecular descriptor (see, e.g., \cite{GF}, \cite{LG}, \cite{LS}, \cite{RS}, \cite{RS0} and the references therein).
During many years, scientists were trying to improve the predictive power of the
Randi\'c index. This led to the introduction of a large number of new topological
descriptors resembling the original Randi\'c index.
The first geometric-arithmetic index $GA_1$, defined in \cite{VF} as
$$
GA_1 = GA_1(G) = \sum_{uv\in E(G)}\frac{\sqrt{d_u d_v}}{\frac12 (d_u + d_v)}
$$
where $uv$ denotes the edge of the graph $G$ connecting the vertices $u$ and $v$, and
$d_u$ is the degree of the vertex $u$,
is one of the successors of the Randi\'c index.
Although $GA_1$ was introduced in $2009$, there are many papers dealing with this index
(see, e.g., \cite{Das10b}, \cite{DGF}, \cite{DGF2}, \cite{MH}, \cite{RS3}, \cite{S}, \cite{VF} and the references therein).
There are other geometric-arithmetic indices, like $Z_{p,q}$ ($Z_{0,1} = GA_1$), but the results in \cite[p.598]{DGF}
show that the $GA_1$ index gathers the
same information on observed molecule as other $Z_{p,q}$ indices.

The reason for introducing a new index is to gain prediction of target property (properties)
of molecules somewhat better than obtained by already presented indices. Therefore,
a test study of predictive power of a new index must be done. As a standard for
testing new topological descriptors, the properties of octanes are commonly used.
We can find 16 physico-chemical properties of octanes at www.moleculardescriptors.eu.

The $GA_1$ index gives better correlation coefficients than $R$ for these properties, but the differences between
them are not significant. However, the predicting ability of the $GA_1$ index compared with
Randi\'c index is reasonably better (see \cite[Table 1]{DGF}).
Although only about 1000 benzenoid hydrocarbons are known, the number of
possible benzenoid hydrocarbons is huge. For instance, the number of
possible benzenoid hydrocarbons with 35 benzene rings is $5.85\cdot 10^{21}$ \cite{NGJ}.
Therefore, the modeling of their physico-chemical properties is very important in order
to predict properties of currently unknown species.
The graphic in \cite[Fig.7]{DGF} (from \cite[Table 2]{DGF}, \cite{TRC}) shows
that there exists a good linear correlation between $GA_1$ and the heat of formation of benzenoid hydrocarbons
(the correlation coefficient is equal to $0.972$).

Furthermore, the improvement in
prediction with $GA_1$ index comparing to Randi\'c index in the case of standard
enthalpy of vaporization is more than 9$\%$. That is why one can think that $GA_1$ index
should be considered in the QSPR/QSAR researches.

The aim of this paper is to obtain new inequalities involving the geometric-arithmetic index $GA_1$
and characterize graphs extremal with respect to them.

Throughout this work, $G=(V (G),E (G))$ denotes a (nonoriented) finite simple (without multiple edges and loops) nontrivial ($E(G) \neq \emptyset$) graph.

\section{Some lower and upper bounds for $GA_1$}

If $G$ is a graph with $m$ edges, minimum degree $\delta$ and maximum degree $\Delta$, then in \cite{Das10b} (see also \cite{DGF}) we find the bounds:
\begin{equation}\label{eq: bound}
\frac{2m\sqrt{\delta \Delta}}{\delta+\Delta} \leq GA_1(G)\leq m.
\end{equation}

Let us recall Lemma 2.2 and Corollary 2.3 in \cite{RS2}.

\begin{lemma}\label{lema 1} Let $f$ be the function $f(t) = \frac{2t}{1+t^2}$ on the interval $[0,\infty)$. Then $f$ strictly
increases in $[0, 1]$, strictly decreases in $[1,\infty)$, $f(t) = 1$ if and only if $t = 1$ and $f(t) = f(t_0)$
if and only if either $t = t_0$ or $t = t_0^{-1}$.
\end{lemma}

\begin{corollary} \label{cor} Let $g$ be the function $g(x, y) = \frac{2\sqrt{xy}}{x+y}$ with $0 < a\leq  x, y \leq b$. Then
\[\frac{2\sqrt{ab}}{a+b}\leq g(x,y)\leq 1.\]
The equality in the lower bound is attained if and only if either $x = a$ and $y = b$, or $x = b$
and $y = a$, and the equality in the upper bound is attained if and only if $x = y$.
\end{corollary}

The following lemma is a direct consequence of Lemma \ref{lema 1} and the fact that $\frac{2\sqrt{xy}}{x+y}= f(t)$ with $t = \sqrt{\frac{x}{y}}$.

\begin{lemma}\label{lema 2} For every $1\leq a< b$ and every $i\in \mathbb{N}$,
\[\frac{2\sqrt{a(a+i)}}{2a+i}< \frac{2\sqrt{b(b+i)}}{2b+i} .\]
\end{lemma}

Let $G$ be a graph with $n$ vertices, $m$ edges, minimum degree $\delta$ and maximum degree $\Delta$.
Let $k=\Delta-\delta$ and consider the partition of the vertices given by their degrees where $V_i$ is the set of vertices with degree $\delta+i$ for every $0\leq i \leq k$. Let $n_i$ be the number of vertices in $V_i$ and $m_{ij}$ be the number of edges joining a vertex in $V_i$ with a vertex in $V_j$. Then,
\begin{equation}\label{eq 1}
GA_1(G)=\sum_{\substack{i,j=0\\ i\leq j}}^k \frac{2 m_{ij} \sqrt{(\delta+i)(\delta+j)}}{2\delta+i+j}= \sum_{i=0}^k m_{ii} +\sum_{\substack{i,j=0\\ i< j}}^k \frac{2m_{ij}\sqrt{(\delta+i)(\delta+j)}}{2\delta+i+j}.
\end{equation}

Therefore, from this and Corollary \ref{cor} it is clear that $GA_1(G)=m$ if and only if all the edges are joining vertices with equal degree.
Hence, $GA_1(G)=m$ if and only if each connected component of $G$ is regular.

\smallskip

As usual, we use the convention
$$
\sum_{\ell\in \emptyset} a_\ell=0.
$$
Therefore, if $k=0$ (i.e., if $G$ is a regular graph), then the last sum in \eqref{eq 1} is equal to zero.

\smallskip

Let us assume $k=\Delta-\delta>0$ and let $n_i=|V_i|$ for every $0\leq i \leq k$.

\begin{proposition}
Let $G$ be a nontrivial graph with minimum degree $\delta$ and maximum degree $\Delta > \delta$. Then
$$
\begin{aligned}
GA_1(G)  &
\leq \sum_{i=0}^k \min\Big\{\frac{1}{2}n_i(\delta+i),\binom{n_i}{2}\Big\}+\sum_{\substack{i,j=0\\ i< j}}^k \frac{2n_i n_j\sqrt{(\delta+i)(\delta+j)}}{2\delta+i+j}
\\
& \leq \sum_{i=0}^k \min\Big\{\frac{1}{2}n_i(\delta+i),\binom{n_i}{2}\Big\}+\sum_{\substack{i,j=0\\ i< j}}^k \frac{2n_i n_j\sqrt{\Delta(\Delta-j+i)}}{2\Delta-j+i}.
\end{aligned}
$$
Furthermore, if $G$ is a connected graph, then we can replace in the previous inequalities $\frac{1}{2}n_i(\delta+i)$ by $\frac{1}{2}n_i(\delta+i)-1$.
\end{proposition}

\begin{proof} First, notice that in every set $V_i$, since there are $n_i$ vertices, $m_{ii} \le \binom{n_i}{2}$.
Also, since $d_v=\delta+i$ for every vertex $v$ in $V_i$, $m_{ii} \le \frac{1}{2}n_i(\delta+i)$.
Moreover, since $V(G)\setminus V_i$ is nonempty, if $G$ is connected, then $m_{ii} \le \frac{1}{2}n_i(\delta+i)-1$.

The number of edges joining $V_i$ and $V_j$ is at most $n_i n_j$.
Thus, the result follows from \eqref{eq 1} and Lemma \ref{lema 2}.
\end{proof}

Note that the hypothesis $\Delta>\delta$ is not essential, since if $\Delta=\delta$ then the graph $G$ is regular and $GA_1(G)=m$.

\smallskip

Let us consider an ordering of the vertices in $G$ where $u<v$ implies that $d_u\leq d_v$.
Let us assume an orientation of the edges where $uv$ is always considered with the orientation given by the ordering $u<v$.
Let $k=\Delta-\delta$, let $m_i$ be the number of oriented edges whose tail is a vertex with degree $\delta+i$ and $m'_i$ the number of oriented edges whose head is a vertex with degree $\delta+i$ for $0\leq i \leq k$. Moreover, let $a_i$ be the number of edges whose tail is a vertex with degree $\delta+i$ and whose head is a vertex with degree at least $\delta+i+1$ with $0\leq i \leq k-1$, let $b_i$ the number of edges whose head is a vertex with degree $\delta+i$ and whose tail is a vertex with degree at most $\delta+i-1$ with $1\leq i \leq k$, and $c_i$ the number of edges joining two vertices with degree $\delta+i$ with $0\leq i \leq k$. Notice that $m_i=a_i+c_i$ and $m'_i=b_i+c_i$ for every $0\leq i \leq k$, $m_k=c_k$ and $m'_0=c_0$.

\smallskip

Define the classes of graphs $\mathcal{G}_1$, $\mathcal{G}_2$ and $\mathcal{G}_3$ as follows.
$\mathcal{G}_1$ is the set of graphs $G$ such that if $uv \in E(G)$, then $d_u=d_v$ or $\max \{d_u,d_v\}= \D$, where $\D$ is the maximum degree of $G$.
$\mathcal{G}_2$ is the set of graphs $G$ such that if $uv \in E(G)$, then $d_u=d_v$ or $\min \{d_u,d_v\}= \d$, where $\d$ is the minimum degree of $G$.
$\mathcal{G}_3$ is the set of graphs $G$ such that if $uv \in E(G)$, then $d_u=d_v$ or $|d_u - d_v|= 1$.

\begin{proposition}\label{Prop:abc}
Let $G$ be a nontrivial graph with minimum degree $\delta$ and maximum degree $\Delta > \delta$. Then
\begin{equation}\label{eq:abc1}
\sum_{i=0}^k c_i+ \sum_{i=0}^{k-1}\frac{2a_i\sqrt{\Delta(\delta+i)}}{\Delta+\delta+i}\leq GA_1(G) \leq \sum_{i=0}^k c_i+ \sum_{i=0}^{k-1}\frac{2a_i\sqrt{(\delta+i)(\delta+i+1)}}{2\delta+2i+1},
\end{equation}
and
\begin{equation}\label{eq:abc2}
\sum_{i=0}^k c_i+ \sum_{i=1}^{k}\frac{2b_i\sqrt{\delta(\delta+i)}}{2\delta+i}\leq GA_1(G) \leq \sum_{i=0}^k c_i+ \sum_{i=1}^{k}\frac{2b_i\sqrt{(\delta+i-1)(\delta+i)}}{2\delta+2i-1}.
\end{equation}
The lower bound in \eqref{eq:abc1} is attained if and only if $G \in \mathcal{G}_1$.
The upper bound in \eqref{eq:abc1} is attained if and only if $G \in \mathcal{G}_3$.
The lower bound in \eqref{eq:abc2} is attained if and only if $G \in \mathcal{G}_2$.
The upper bound in \eqref{eq:abc2} is attained if and only if $G \in \mathcal{G}_3$.
\end{proposition}

\begin{proof}
Since
\[1<\frac{\delta+i+1}{\delta+i}\leq \frac{\delta+i+r}{\delta+i} \leq \frac{\Delta}{\delta+i}\]
for every $1\leq r \leq \Delta-\delta-i$ and $f$ is decreasing on $[1,\infty)$,
Lemma \ref{lema 1} gives
\[f\left( \sqrt{\frac{\delta+i+1}{\delta+i}}\,\right) \geq f\left( \sqrt{\frac{\delta+i+r}{\delta+i}}\,\right) \geq f\left( \sqrt{\frac{\Delta}{\delta+i}}\, \right)\! .\]
Hence, \eqref{eq 1} gives \eqref{eq:abc1}.

Since
\[1< \frac{\delta+i}{\delta+i-1} \leq \frac{\delta+i}{\delta+i-r} \leq \frac{\delta+i}{\delta} \]
for every $1\leq r \leq i$ and $f$ is decreasing on $[1,\infty)$,
Lemma \ref{lema 1} gives
\[f\left( \sqrt{\frac{\delta+i}{\delta+i-1}}\,\right) \geq f\left( \sqrt{\frac{\delta+i}{\delta+i-r}}\,\right) \geq f\left( \sqrt{\frac{\delta +i}{\delta}}\,\right)\! .\]
Therefore, \eqref{eq:abc2} follows from \eqref{eq 1}.

One can easily check the statements on the equalities.
\end{proof}

\begin{remark} Note that if $C:=\sum_{i=0}^k c_i$, $r_i:=2a_i\sqrt{\delta+i}$ and $r'_i:=2b_i\sqrt{\delta+i}$, then

\[C+ \sum_{i=0}^{k-1} r_i \frac{ \sqrt{\Delta}}{\Delta+\delta+i}\leq GA_1(G) \leq C+ \sum_{i=0}^{k-1} r_i \frac{\sqrt{\delta+i+1}}{2\delta+2i+1},\]
and
\[C+ \sum_{i=1}^{k} r'_i \frac{\sqrt{\delta}}{2\delta+i}\leq GA_1(G) \leq C+ \sum_{i=1}^{k} r'_i \frac{\sqrt{\delta+i-1}}{2\delta+2i-1}.\]
\end{remark}

Define the classes of graphs $\mathcal{G}_1^0$ and $\mathcal{G}_2^0$ as follows.
$\mathcal{G}_1^0$ is the set of graphs $G$ such that if $uv \in E(G)$, then $\max \{d_u,d_v\}= \D$, where $\D$ is the maximum degree of $G$.
$\mathcal{G}_2^0$ is the set of graphs $G$ such that if $uv \in E(G)$, then $\min \{d_u,d_v\}= \d$, where $\d$ is the minimum degree of $G$.
It is clear that $\mathcal{G}_1^0 \subset \mathcal{G}_1$ and $\mathcal{G}_2^0 \subset \mathcal{G}_2$.

\begin{corollary}\label{cor: m}
Let $G$ be a nontrivial graph with minimum degree $\delta\geq 2$ and maximum degree $\Delta > \delta$. Then
\[\sum_{i=0}^{k}\frac{2m_i\sqrt{\Delta(\delta+i)}}{\Delta+\delta+i}=\sum_{i=0}^{k-1}\frac{2m_i\sqrt{\Delta(\delta+i)}}{\Delta+\delta+i} +m_k \leq GA_1(G) \leq m,\]
and
\[\sum_{i=0}^{k}\frac{2m'_i\sqrt{\delta(\delta+i)}}{2\delta+i}=m'_0+\sum_{i=1}^{k}\frac{2m'_i\sqrt{\delta(\delta+i)}}{2\delta+i}\leq GA_1(G) \leq m.\]
The first (respectively, second) lower bound is attained if and only if $G \in \mathcal{G}_1^0$ (respectively, $G \in \mathcal{G}_2^0$).
\end{corollary}

Since in a connected graph with at least $3$ vertices, there are no edges joining two vertices with degree 1, we have the following consequence.

\begin{corollary}
Let $G$ be a nontrivial connected graph with at least $3$ vertices, minimum degree $1$ and maximum degree $\Delta$. Then
\[\sum_{i=0}^{k}\frac{2m_i\sqrt{\Delta(i+1)}}{\Delta+i+1}\leq GA_1(G) \leq \frac{2\sqrt{2}\,m_0}{3}+ \sum_{i=1}^{k}m_i = \frac{2\sqrt{2}\,m_0}{3}+m-m_0.\]
\end{corollary}

Similarly, the following result, which is Corollary 3.11 in \cite{Das10b}, is an immediate consequence from Corollary \ref{cor: m}.
A vertex $v$ is called \emph{pendant} if the set of its neighbors has exactly one vertex, this is, if $d_v=1$.
Thus, with the notation above, there are $m_0$ pendant vertices.

\begin{corollary}
Let $G$ be a nontrivial connected graph with at least $3$ vertices, minimum degree $1$ and minimal non-pendant vertex degree $\delta_1$. Then
\[ GA_1(G) \leq \frac{2m_0\sqrt{\delta_1}}{\delta_1+1}+ m-m_0.\]
\end{corollary}

Given any graph $G$ and $uv\in E(G)$, let us define the \emph{gradient} of the edge $uv$ as $\nabla_{uv}:=|d_u-d_v|$.

\begin{proposition}
Let $G$ be a nontrivial graph with minimum degree $\delta$ and maximum degree $\Delta$.
If $d= \min_{uv\in E(G)}\nabla_{uv}$ and $D= \max_{uv\in E(G)}\nabla_{uv}$, then
\begin{equation} \label{eq:p1}
\frac{2m\sqrt{\delta(\delta+D)}}{2\delta+D} \leq GA_1(G)\leq \frac{2m\sqrt{(\Delta-d)\Delta}}{2\Delta -d}.
\end{equation}
The equality in each inequality is attained if and only if $G$ is
either regular or bipartite with the two sets being respectively the set of vertices with degree $\delta$ and degree $\Delta$.
\end{proposition}

\begin{proof}
Consider any edge $uv \in E(G)$.
By symmetry, we can assume that $d_v \ge d_u$.
Thus, $d\leq d_v-d_u\leq D$ and
$$
\d d_v
\le \d d_u + \d D
\le \d d_u + d_u D.
$$
Hence,
$$
\frac{d_v}{d_u} \leq \frac{\delta+D}{\delta}
$$
with equality if and only if $d_v=d_u + D$ and $d_u= \d$.
Since
$$
\D d_u
\le \D d_v - \D d
\le \D d_v - d_v d,
$$
we have
$$
\frac{\Delta}{\Delta-d}\leq \frac{d_v}{d_u}
$$
with equality if and only if $d_u=d_v - d$ and $d_v= \D$.
Hence,
$$
1\leq \frac{\Delta}{\Delta-d}\leq \frac{d_v}{d_u} \leq \frac{\delta+D}{\delta},
$$
and Lemma \ref{lema 1} gives
\begin{equation} \label{eq:p2}
f\left( \sqrt{\frac{\Delta}{\Delta-d}}\, \right) \geq f\left( \sqrt{\frac{d_v}{d_u}}\, \right) \geq f\left( \sqrt{\frac{\delta+D}{\delta}}\, \right)\!.
\end{equation}
We obtain the inequalities in \eqref{eq:p1} by adding \eqref{eq:p2} for every $uv \in E(G)$.

Therefore, the equality in the lower bound is attained if and only if $d_v=d_u + D$ and $d_u= \d$ for every $uv \in E(G)$ with $d_v \ge d_u$;
the equality in the upper bound is attained if and only if $d_u=d_v - d$ and $d_v= \D$ for every $uv \in E(G)$ with $d_v \ge d_u$.
Hence, the equality in each inequality is attained if and only if $G$ is
either regular (if $D=0$) or bipartite with the two sets being respectively the set of vertices with degree $\delta$ and degree $\Delta$.
\end{proof}

Let $E_0,\dots,E_k$ (with $k=\Delta-\delta$) be a partition of the edges of $G$ given by the gradient where $e\in E_i$ if $\nabla_e=i$ for each $0\leq i \leq k$.
Let $e_i$ be the number of edges in $E_i$.

\begin{proposition}
Let $G$ be a nontrivial graph with minimum degree $\delta$ and maximum degree $\Delta$. Then
\begin{equation} \label{eq:p3}
\sum_{i=0}^k \frac{2e_i\sqrt{\delta(\delta+i)}}{2\delta+i} \leq GA_1(G) \leq \sum_{i=0}^k \frac{2e_i\sqrt{\Delta(\Delta-i)}}{2\Delta-i}.
\end{equation}
The upper (respectively, lower) bound is attained if and only if $G \in \mathcal{G}_1^0$ (respectively, $G \in \mathcal{G}_2^0$).
\end{proposition}

\begin{proof}
Consider any edge $uv \in E_i$.
By symmetry, we can assume that $d_v-d_u=i$.
Since $i d_v \le i\D$, we have
$$
\D d_u = \D( d_v-i) \le \D d_v -i d_v.
$$
Hence,
$$
\frac{\Delta}{\Delta-i}\leq \frac{d_v}{d_u}
$$
with equality if and only if $d_v= \D$.
Since $i \d \le i d_u$,
$$
\d d_v = \d( d_u+i) \le \d d_u+i d_u,
$$
and we have
$$
\frac{d_v}{d_u} \leq \frac{\delta+i}{\delta}
$$
with equality if and only if $d_u= \d$.
Therefore,
$$
1\leq \frac{\Delta}{\Delta-i}\leq \frac{d_v}{d_u} \leq \frac{\delta+i}{\delta},
$$
and Lemma \ref{lema 1} gives
\begin{equation} \label{eq:p4}
f\left( \sqrt{\frac{\Delta}{\Delta-i}}\, \right) \geq f\left( \sqrt{\frac{d_v}{d_u}}\, \right) \geq f\left( \sqrt{\frac{\delta+i}{\delta}}\, \right)\!.
\end{equation}
We obtain the inequalities in \eqref{eq:p3} by adding \eqref{eq:p4} for every $uv \in E(G)$.

Therefore, the equality in the lower bound is attained if and only if $d_u= \d$ for every $uv \in E(G)$ with $d_v \ge d_u$;
the equality in the upper bound is attained if and only if $d_v= \D$ for every $uv \in E(G)$ with $d_v \ge d_u$.
Hence, the upper (respectively, lower) bound is attained if and only if $G \in \mathcal{G}_1^0$ (respectively, $G \in \mathcal{G}_2^0$).
\end{proof}

\begin{remark} \label{remark bipartite}
Therefore, notice that $GA_1(G)=\frac{2m\sqrt{\delta \Delta}}{\delta+\Delta}$ if and only if $\nabla_{uv}=\Delta-\delta$ for every edge $uv$. Furthermore, if $\delta>0$, this occurs if and only if the graph is
either regular or bipartite with the two sets being respectively the set of vertices with degree $\delta$ and degree $\Delta$.
\end{remark}

\section{Bounds involving other topological indices}

In \cite[Lemma 3]{S} appears the following result.

\begin{lemma} \label{l:h}
Let $h$ be the function $h(x,y)=\frac{2xy}{x+y}$ with $\d\le x,y \le \D$. Then
$
\d \le h(x,y) \le \D.
$
The lower (respectively, upper) bound is attained if and only if $x=y=\d$ (respectively, $x=y=\D$).
\end{lemma}

First, we obtain a lower bound of $GA_1(G)$ depending on $n$, $m$ and $\d$.

\begin{proposition} \label{t:end}
We have for any graph $G$ with minimum degree $\d$, $n$ vertices and $m$ edges
$$
GA_1(G) \ge \frac{2m\sqrt{(n-1)\d}}{n+\d-1}\,,
$$
and the equality is attained if and only if $G$ is either a complete graph or a star graph.
\end{proposition}

\begin{proof}
Recall that $\d \le d_u \le n-1$ for every $u\in V(G)$.
By Corollary \ref{cor}, taking $a=\d$ and $b=n-1$, we have
$$
GA_1(G) = \sum_{uv\in E(G)}\frac{2 \sqrt{d_u d_v}}{d_u + d_v}
\ge \sum_{uv\in E(G)}\frac{2\sqrt{(n-1) \d}}{n+\d-1}
= \frac{2m\sqrt{(n-1)\d}}{n+\d-1} \,.
$$
By Corollary \ref{cor}, the equality holds for $G$ if and only if every edge joins a vertex of degree $\d$ with a vertex of degree $n-1$;
if $\d=n-1$, then this holds if and only if $G$ is a complete graph;
if $\d<n-1$, then this holds if and only if $\d=1$ and $G$ is a star graph.
\end{proof}

In what follows we will need Cassels inequality \cite[Appendix 1]{Watson}.
Although it is a well-known result, it is not easy to find the characterization of the cases of equality.
For the sake of completeness, we prove here a more general statement (following the argument of Niculescu \cite{Niculescu})
that allows to characterize the equality.

\begin{lemma} \label{l:PS1}
Let $(X,\mu)$ be a measure space and $f,g : X \rightarrow \RR$ non-negative measurable functions.
If $\o f \le g \le \O f$ $\mu$-a.e. for positive constants $0<\o \le \O$, then
\begin{equation} \label{eq:PS1}
\Big(\int_X f^2 \,d\mu \Big)^{1/2} \Big(\int_X g^2 \,d\mu \Big)^{1/2}
\le \frac12 \left(\sqrt{\frac{\O}{\o}}+ \sqrt{\frac{\o}{\O}} \,\right) \int_X fg \,d\mu
\end{equation}
and the equality is attained if and only if we have $\o = \O$ or $f = g = 0$ $\mu$-a.e.
\end{lemma}

\begin{proof}
Recall that
$$
\frac1{\e} a^2+\e b^2 \ge 2ab
$$
and the equality holds if and only if $a=\e b$.
Therefore, the hypotheses imply
$$
\begin{aligned}
0
\ge \int_X (g-\o f)(g-\O f) \, d\mu
& = \int_X g^2 \, d\mu -(\O+\o)\int_X fg \, d\mu + \O\o\int_X f^2 \, d\mu
\\
\left(\sqrt{\frac{\O}{\o}}+ \sqrt{\frac{\o}{\O}} \,\right) \int_X fg \,d\mu
& \ge \frac{1}{\sqrt{\O\o}}\int_X g^2 \, d\mu + \sqrt{\O\o}\int_X f^2 \, d\mu
\ge 2 \Big(\int_X g^2 \,d\mu \Big)^{1/2} \Big(\int_X f^2 \,d\mu \Big)^{1/2}.
\end{aligned}
$$

Furthermore, the equality in \eqref{eq:PS1} holds if and only if
$(g-\o f)(g-\O f)=0$ $\mu$-a.e. and $\int_X g^2 \,d\mu = \O\o \int_X f^2 \,d\mu$.
If $\o=\O$, then $g=\o f$ and both equalities hold.
Assume now that $\o<\O$.
Since $f,g \ge 0$ and
$$
\int_X g^2 \,d\mu = \O\o \int_X f^2 \,d\mu
\quad \;
\Leftrightarrow
\quad
\int_X \big(g-\sqrt{\O\o}\, f\big)\big(g+\sqrt{\O\o}\, f\big) \, d\mu=0,
$$
the equality $\int_X g^2 \,d\mu = \O\o \int_X f^2 \,d\mu$ is equivalent to
$g=\sqrt{\O\o}\, f$ $\mu$-a.e.
Thus,
$$
0 = (g-\o f)(g-\O f) = \big(\sqrt{\O\o}-\o)\big(\sqrt{\O\o}-\O \big)f^2
$$
and we conclude that if $\o<\O$, then the equality in \eqref{eq:PS1} is attained if and only if $f = g = 0$ $\mu$-a.e.
\end{proof}

We have the following direct consequence.

\begin{lemma} \label{l:PS2}
If $a_j,b_j\ge 0$ and $\o b_j \le a_j \le \O b_j$ for $1\le j \le k$, then
$$
\Big(\sum_{j=1}^k a_j^2 \Big)^{1/2} \Big(\sum_{j=1}^k b_j^2 \Big)^{1/2}
\le \frac12 \left(\sqrt{\frac{\O}{\o}}+ \sqrt{\frac{\o}{\O}} \,\right)\sum_{j=1}^k a_jb_j .
$$
If $a_j>0$ for some $1\le j \le k$, then the equality holds if and only if $\o=\O$ and $a_j=\o b_j$ for every $1\le j \le k$.
\end{lemma}

Recall that the \emph{variable Zagreb index} is defined in \cite{MN} as
$$
Z_\a(G) = \sum_{uv\in E(G)} (d_u d_v)^\a .
$$
The variable Zagreb index was used in the structure-boiling point modeling of benzenoid hydrocarbons.
The obtained model is practically identical to the model based on the variable vertex-connectivity index
and this is due to close relationship between the formulas for the two indices.
Note that $Z_{-1/2}$ is the usual Randi\'c index, $Z_{1}$ is the second Zagreb index $M_2$,
$Z_{-1}$ is the modified Zagreb index \cite{NKMT}, etc.

\begin{theorem} \label{t:lb5}
We have for any graph $G$ with minimum degree $\d$, maximum degree $\D$ and $m$ edges, and $\a \in \RR$
$$
\frac{c_{1,\a} m^2}{Z_{\a}(G)}
\le GA_1(G)
\le \frac{c_{2,\a} m^2}{Z_{\a}(G)} \,,
$$
with
$$
c_{1,\a}:=
\left\{
\begin{array}{ll}
\d^{2\a+1} \D^{-1}, \quad &\mbox{if }\ \a\ge -1/2,\\
\D^{2\a}, \quad &\mbox{if }\ \a\le -1/2,
\end{array}
\right.
\qquad
c_{2,\a}:=
\left\{
\begin{array}{ll}
\frac{\D(\D^{2\a}+\d^{2\a})^2}{4 \d^{2\a+1}}\,, \quad &\mbox{if }\ \a\ge -1/2,\\
\frac{(\D^{2\a}+\d^{2\a})^2}{4 \D^{2\a}}\,, \quad &\mbox{if }\ \a\le -1/2,
\end{array}
\right.
$$
and each inequality is attained for some fixed $\a$ if and only if $G$ is regular.
\end{theorem}

\begin{proof}
Cauchy-Schwarz inequality gives
$$
m^2
= \Big( \sum_{uv\in E(G)} (d_u d_v)^{\a/2} (d_u d_v)^{-\a/2} \Big)^2
\le \sum_{uv\in E(G)} (d_u d_v)^{\a} \sum_{uv\in E(G)} (d_u d_v)^{-\a}
= Z_{\a}(G) \sum_{uv\in E(G)} (d_u d_v)^{-\a}.
$$
We have
$$
GA_1(G)
= \sum_{uv\in E(G)}\frac{2\sqrt{d_u d_v}}{d_u + d_v}
\ge \frac1{\D}\sum_{uv\in E(G)} (d_u d_v)^{\a+1/2}(d_u d_v)^{-\a}
\,.
$$
If $\a\le -1/2$, then
$$
GA_1(G)
\ge \frac1{\D}\sum_{uv\in E(G)} (d_u d_v)^{\a+1/2}(d_u d_v)^{-\a}
\ge \D^{2\a}\sum_{uv\in E(G)} (d_u d_v)^{-\a}
\ge \frac{\D^{2\a} m^2}{Z_{\a}(G)} \,.
$$
If $\a\ge -1/2$, then
$$
GA_1(G)
\ge \frac{1}{\D} \sum_{uv\in E(G)} (d_u d_v)^{\a+1/2}(d_u d_v)^{-\a}
\ge \frac{\d^{2\a+1}}{\D} \sum_{uv\in E(G)} (d_u d_v)^{-\a}
\ge \frac{\d^{2\a+1} m^2}{\D Z_{\a}(G)} \,.
$$
Hence, we obtain
$$
\frac{c_{1,\a} m^2}{Z_{\a}(G)}
\le GA_1(G).
$$

Since
$$
\begin{aligned}
\d^{2\a} \le \frac{(d_u d_v)^{\a/2}}{(d_u d_v)^{-\a/2}}=(d_u d_v)^{\a} \le \D^{2\a},
& \qquad
\text{ if }\,\a\ge 0,
\\
\D^{2\a} \le \frac{(d_u d_v)^{\a/2}}{(d_u d_v)^{-\a/2}}=(d_u d_v)^{\a} \le \d^{2\a},
& \qquad
\text{ if }\,\a\le 0,
\end{aligned}
$$
Lemma \ref{l:PS2} gives
$$
\begin{aligned}
m^2
& = \Big( \sum_{uv\in E(G)} (d_u d_v)^{\a/2} (d_u d_v)^{-\a/2} \Big)^2
\ge \frac{\sum_{uv\in E(G)} (d_u d_v)^{\a} \sum_{uv\in E(G)} (d_u d_v)^{-\a}}{\frac14\Big(\frac{\D^{\a}}{\d^{\a}}+\frac{\d^{\a}}{\D^{\a}} \Big)^2}
\\
& = \frac{4\,\D^{2\a}\d^{2\a}}{(\D^{2\a}+\d^{2\a})^2} \,\,Z_{\a}(G) \!\! \sum_{uv\in E(G)} (d_u d_v)^{-\a}.
\end{aligned}
$$
If $\a\le -1/2$, then
$$
\begin{aligned}
GA_1(G)
& \le \frac1{\d}\sum_{uv\in E(G)} (d_u d_v)^{\a+1/2}(d_u d_v)^{-\a}
\le \d^{2\a} \sum_{uv\in E(G)} (d_u d_v)^{-\a}
\\
& \le \d^{2\a} \frac{(\D^{2\a}+\d^{2\a})^2}{4\,\D^{2\a}\d^{2\a}}\, \frac{m^2}{Z_{\a}(G)}
= \frac{(\D^{2\a}+\d^{2\a})^2}{4\, \D^{2\a}}\, \frac{m^2}{Z_{\a}(G)} \,.
\end{aligned}
$$
If $\a\ge -1/2$, then
$$
\begin{aligned}
GA_1(G)
& \le \frac1{\d}\sum_{uv\in E(G)} (d_u d_v)^{\a+1/2}(d_u d_v)^{-\a}
\le \frac{\D^{2\a+1}}{\d}\sum_{uv\in E(G)} (d_u d_v)^{-\a}
\\
& \le \frac{\D^{2\a+1}}{\d}\, \frac{(\D^{2\a}+\d^{2\a})^2}{4\,\D^{2\a}\d^{2\a}}\, \frac{m^2}{Z_{\a}(G)}
= \frac{\D(\D^{2\a}+\d^{2\a})^2}{4\,\d^{2\a+1}}\, \frac{m^2}{Z_{\a}(G)} \,.
\end{aligned}
$$
Hence, we obtain
$$
 GA_1(G)
\le \frac{c_{2,\a} m^2}{Z_{\a}(G)} \,,
$$

If the graph is regular, then
$c_{1,\a}=c_{2,\a}=\D^{2\a}$,
the lower and upper bounds are the same, and they are equal to $GA_1(G)$.
If a bound is attained for some $\a$, then we have either
$\frac{d_u+d_v}2 = \D$ for every $uv\in E(G)$ or $\frac{d_u+d_v}2 = \d$ for every $uv\in E(G)$ and we conclude that $d_u = d_v$ for every $u,v\in V(G)$.
\end{proof}

\begin{corollary} \label{c:lb18}
We have for any graph $G$ with minimum degree $\d$, maximum degree $\D$ and $m$ edges
$$
\frac{\d^{3} m^2}{\D M_{2}(G)}
\le GA_1(G)
\le \frac{\D(\D^{2}+\d^{2})^2 m^2}{4\, \d^{3} M_{2}(G)} \,,
$$
and each inequality is attained if and only if $G$ is regular.
\end{corollary}

With motivation from the Randi\'c, Zagreb and harmonic indices, the \emph{general sum-connectivity index} $H_{\a}$
was defined by Zhou and Trinajsti\'c in \cite{ZT2} as
$$
H_{\a}(G) = \sum_{uv\in E(G)} (d_u + d_v)^{\a},
$$
with $\a \in \RR$.
Note that $H_{1}$ is the first Zagreb index $M_1$, $2H_{-1}$ is the harmonic index $H$,
$H_{-1/2}$ is the sum-connectivity index, etc.

\begin{theorem} \label{t:mh}
We have for any graph $G$ with minimum degree $\d$ and maximum degree $\D$
$$
\frac{4\,\D \d \sqrt{M_2(G)H_{-2}(G)}}{\D^{2}+\d^{2}}
\le GA_1(G)
\le 2 \sqrt{M_2(G)H_{-2}(G)} \,.
$$
The equality in the lower bound is attained if and only if $G$ is regular.
The equality in the upper bound is attained if and only if there exists a constant $\l$ such that $d_u d_v (d_u + d_v)^2 = \l$ for every $uv\in E(G)$.
\end{theorem}

\begin{proof}
Cauchy-Schwarz inequality gives
$$
GA_1(G)^2 =
\Big( \sum_{uv\in E(G)} \frac{2\sqrt{d_u d_v}}{d_u + d_v} \;\Big)^2
\le \sum_{uv\in E(G)} 4d_u d_v \sum_{uv\in E(G)} \frac{1}{(d_u + d_v)^2}
= 4 M_2(G)H_{-2}(G).
$$

Since
$$
4\d^{2} \le \frac{2\sqrt{d_u d_v}}{(d_u + d_v)^{-1}}=2\sqrt{d_u d_v}\,(d_u + d_v) \le 4\D^{2},
$$
Lemma \ref{l:PS2} gives
$$
\begin{aligned}
GA_1(G)^2
& = \Big( \sum_{uv\in E(G)} \frac{2\sqrt{d_u d_v}}{d_u + d_v} \;\Big)^2
\ge \frac{\sum_{uv\in E(G)} 4d_u d_v \sum_{uv\in E(G)} \frac{1}{(d_u + d_v)^2}}{\frac14 \Big(\frac{\D}{\d}+\frac{\d}{\D} \Big)^2}
\\
& = \frac{16\,\D^{2}\d^{2}M_2(G)H_{-2}(G)}{(\D^{2}+\d^{2})^2} \,.
\end{aligned}
$$

If the graph is regular, then the lower and upper bounds are the same, and they are equal to $GA_1(G)$.

If the lower bound is attained, then Lemma \ref{l:PS2} gives that $4\d^2 = 4\D^2$ and $G$ is regular.

If the upper bound is attained, then Cauchy-Schwarz inequality gives that
$$
\frac{2\sqrt{d_u d_v}}{(d_u + d_v)^{-1}}=2\sqrt{d_u d_v}\,(d_u + d_v)
$$
is constant, and so there exists a constant $\l$ such that $d_u d_v (d_u + d_v)^2 = \l$ for every $uv\in E(G)$.
\end{proof}

We say that a graph is $(\a,\b)$-biregular if it is a bipartite graph for which any vertex in one side of the given bipartition has degree $\a$
and any vertex in the other side of the bipartition has degree $\b$.

\medskip

The following result characterizes in many cases the equality in the upper bound in Theorem \ref{t:mh}.

\begin{proposition} \label{p:mh}
Let $G$ be a graph.

\begin{itemize}
\item
If there exists a constant $\l$ such that $d_u d_v (d_u + d_v)^2 = \l$ for every $uv\in E(G)$, then each connected component of $G$ is either regular or biregular.
\item
If $G$ is a connected graph, then there exists a constant $\l$ such that $d_u d_v (d_u + d_v)^2 = \l$ for every $uv\in E(G)$ if and only if $G$ is either regular or biregular.
\end{itemize}
\end{proposition}

\begin{proof}
Assume that there exists a constant $\l$ such that $d_u d_v (d_u + d_v)^2 = \l$ for every $uv\in E(G)$.
Since the function $f:[0,\infty)\times [0,\infty) \rightarrow \RR$ defined as $f(x,y)=xy (x + y)^2$
is strictly increasing in $y$ for each fixed $x$,
given any vertex $u\in V(G)$, every neighbor of $u$ has the same degree.
Hence, each connected component of $G$ is either regular or biregular.
Furthermore, if $G$ is connected, then $d_u d_v (d_u + d_v)^2 = \l$ for every $uv\in E(G)$ if and only if $G$ is regular or biregular.
\end{proof}

\begin{example} It may be wondered if there exist two different pairs of natural numbers $a,b$ and $c,d$ such that $ab(a+b)^2=cd(c+d)^2$. The answer is affirmative and such pairs of numbers can be obtained as follows.

First let us choose two Pythagorean triples: $\a_1,\b_1,\g_1$ and $\a_2,\b_2,\g_2$ with $\a_1\b_1=\a_2\b_2$ (e.g., $12,35,37$ and $20,21,29$) 
and let $a=\g_2\a_1^2$, $b=\g_2\b_1^2$, $c=\g_1\a_2^2$ and $d=\g_1\b_2^2$. 
Then, notice that
\[\g_2\a_1^2 \g_2\b_1^2(\g_2\a_1^2 + \g_2\b_1^2)^2= \a_1^2\b_1^2\g_1^4\g_2^4 =\lambda,\]
and
\[\g_1\a_2^2 \g_1\b_2^2(\g_1\a_2^2 + \g_1\b_2^2)^2= \a_2^2\b_2^2\g_1^4\g_2^4 = \a_1^2\b_1^2\g_1^4\g_2^4 =\lambda.\]
Therefore, the best characterization of the upper bound in Theorem \ref{t:mh} is the one in Proposition \ref{p:mh}.
\end{example}

In \cite[Theorem 4]{S} appears the inequality
$$
GA_1(G)
\le \sqrt{M_2(G)Z_{-1}(G)} \,.
$$
Note that Theorem \ref{t:mh} improves this upper bound of $GA_1(G)$ since $4d_ud_v \le (d_u+d_v)^2$ gives
$$
4 H_{-2}(G)
= \sum_{uv\in E(G)} \frac{4}{(d_u + d_v)^2}
\le \sum_{uv\in E(G)} \frac{1}{d_u d_v}
= Z_{-1}(G),
$$
and
$2 \sqrt{H_{-2}(G)} \le \sqrt{Z_{-1}(G)}$.

\begin{theorem} \label{t:mm2}
We have for any graph $G$ with minimum degree $\d$, maximum degree $\D$ and $m$ edges
$$
\frac{\d^2 m^2}{ M_2(G)}
\le GA_1(G)
\le \frac{\D^{1/2} (\D + \d\,)^3 m^2}{8 \,\d^{3/2}M_2(G)} \,,
$$
and each equality is attained if and only if $G$ is regular.
\end{theorem}

\begin{proof}
Lemma \ref{l:h}, Cauchy-Schwarz inequality and Corollary \ref{cor} give
$$
\begin{aligned}
\big( \d m \big)^2
& \le
\Big( \sum_{uv\in E(G)} \frac{2d_u d_v}{d_u + d_v} \;\Big)^2
\le \sum_{uv\in E(G)} \Big( \, \frac{2\sqrt{d_u d_v}}{d_u + d_v}\;\Big)^2 \sum_{uv\in E(G)} \Big(\sqrt{d_u d_v}\;\Big)^2
\\
& \le \sum_{uv\in E(G)} \frac{2\sqrt{d_u d_v}}{d_u + d_v} \sum_{uv\in E(G)} d_u d_v
= GA_1(G) M_2(G) .
\end{aligned}
$$

Since
$$
\frac{1}{\D} \le \frac{\frac{2\sqrt{d_u d_v}}{d_u + d_v}}{\sqrt{d_u d_v}} =\frac{2}{d_u + d_v} \le \frac{1}{\d}\,,
$$
Lemmas \ref{l:h}, \ref{l:PS2} and Corollary \ref{cor} give
$$
\begin{aligned}
\big( \D m \big)^2
& \ge
\Big( \sum_{uv\in E(G)} \frac{2d_u d_v}{d_u + d_v} \;\Big)^2
\ge \frac{\sum_{uv\in E(G)} \Big( \, \frac{2\sqrt{d_u d_v}}{d_u + d_v}\;\Big)^2 \sum_{uv\in E(G)} \Big(\sqrt{d_u d_v}\;\Big)^2}{\frac14 \Big(\sqrt{\frac{\D}{\d}}+\sqrt{\frac{\d}{\D}}\, \Big)^2}
\\
& \ge \frac{\frac{2\sqrt{\D\d}}{\D + \d}\sum_{uv\in E(G)} \frac{2\sqrt{d_u d_v}}{d_u + d_v} \sum_{uv\in E(G)} d_u d_v}{\frac14 \Big(\sqrt{\frac{\D}{\d}}+\sqrt{\frac{\d}{\D}}\, \Big)^2}
= \frac{8\,(\D\d)^{3/2}GA_1(G) M_2(G)}{(\D + \d\,)^3}\, .
\end{aligned}
$$

If the graph is regular, then the lower and upper bounds are the same, and they are equal to $GA_1(G)$.

By Lemma \ref{l:h}, if a bound is attained, then we have either
$d_u=d_v = \d$ for every $uv\in E(G)$ or $d_u=d_v = \D$ for every $uv\in E(G)$, and we conclude that $d_u = d_v$ for every $u,v\in V(G)$.
\end{proof}

Note that Theorem \ref{t:mm2} improves the bounds in Corollary \ref{c:lb18}, since
$$
\frac{\d^{3}}{\D}
\le \d^{2},
\qquad
\frac{\D^{1/2} (\D + \d\,)^3}{8 \,\d^{3/2}}
\le \frac{\D(\D^{2}+\d^{2})^2}{4\, \d^{3}} \,,
$$
where the second inequality follows from
$$
\begin{aligned}
(s-1)(2s^8+2s^7+2s^6+s^5 & +5s^4+2s^3+2s^2-s+1)\ge 0  \qquad \text{for } \, s\ge 1,
\\
2s^9-s^6+4s^5-3s^4-3s^2 & +2s-1 \ge 0  \qquad \text{for } \, s\ge 1,
\\
(s^2+1)^3 \le & \; 2s(s^4+1)^2  \qquad \text{for } \, s\ge 1,
\\
(t+1)^3 \le & \; 2\sqrt{t}\,(t^2+1)^2  \qquad \text{for } \, t\ge 1,
\\
 \d^{3/2}(\D + \d\,)^3
& \le 2\D^{1/2}(\D^{2}+\d^{2})^2 \qquad \text{taking } \, t=\frac{\Delta}{\delta}.
\end{aligned}
$$

\begin{theorem} \label{t:hz}
We have for any graph $G$ with minimum degree $\d$ and maximum degree $\D$
$$
\frac{H(G)^2}{ Z_{-1}(G)}
\le GA_1(G)
\le \frac{(\D + \d\,)^3H(G)^2}{8\,(\D\d)^{3/2}Z_{-1}(G)} \,,
$$
and each inequality is attained if and only if $G$ is regular.
\end{theorem}

\begin{proof}
Cauchy-Schwarz inequality and Corollary \ref{cor} give
$$
\begin{aligned}
H(G)^2
& =
\Big( \sum_{uv\in E(G)} \frac{2}{d_u + d_v} \;\Big)^2
\le \sum_{uv\in E(G)} \Big( \, \frac{2\sqrt{d_u d_v}}{d_u + d_v}\;\Big)^2 \sum_{uv\in E(G)} \Big(\frac1{\sqrt{d_u d_v}}\;\Big)^2
\\
& \le \sum_{uv\in E(G)} \frac{2\sqrt{d_u d_v}}{d_u + d_v} \sum_{uv\in E(G)} \frac1{d_u d_v}
= GA_1(G) Z_{-1}(G) .
\end{aligned}
$$

Since Lemma \ref{l:h} implies
$$
\d \le \frac{\frac{2\sqrt{d_u d_v}}{d_u + d_v}}{\frac{1}{\sqrt{d_u d_v}}} =\frac{2 d_u d_v}{d_u + d_v} \le \D,
$$
Lemma \ref{l:PS2} gives
$$
\begin{aligned}
H(G)^2
& =
\Big( \sum_{uv\in E(G)} \frac{2}{d_u + d_v} \;\Big)^2
\ge \frac{\sum_{uv\in E(G)} \Big( \, \frac{2\sqrt{d_u d_v}}{d_u + d_v}\;\Big)^2 \sum_{uv\in E(G)} \Big(\frac{1}{\sqrt{d_u d_v}}\;\Big)^2}{\frac14 \Big(\sqrt{\frac{\D}{\d}}+\sqrt{\frac{\d}{\D}}\, \Big)^2}
\\
& \ge \frac{\frac{2\sqrt{\D\d}}{\D + \d}\sum_{uv\in E(G)} \frac{2\sqrt{d_u d_v}}{d_u + d_v} \sum_{uv\in E(G)} \frac{1}{d_u d_v}}{\frac14 \Big(\sqrt{\frac{\D}{\d}}+\sqrt{\frac{\d}{\D}}\, \Big)^2}
= \frac{8\,(\D\d)^{3/2}GA_1(G) Z_{-1}(G)}{(\D + \d\,)^3}\, .
\end{aligned}
$$

If the graph is regular, then the lower and upper bounds are the same, and they are equal to $GA_1(G)$.

By Lemma \ref{l:PS2}, if the upper bound is attained, then $\D=\d$ and $G$ is regular.

If the lower bound is attained, then Corollary \ref{cor} gives $d_u = d_v$ for every $uv\in E(G)$.
Cauchy-Schwarz inequality gives that there exists a constant $\l$ such that
$$
\frac{2\sqrt{d_u d_v}}{d_u + d_v} = \l \frac1{\sqrt{d_u d_v}}
$$
for every $uv\in E(G)$.
Hence, $d_u = \l$ for every $u\in V(G)$ and $G$ is regular.
\end{proof}

The \emph{forgotten topological index} is defined as $F(G) = \sum_{u\in V(G)} d_u^3 = \sum_{uv\in E(G)} (d_u^2+d_v^2)$ (see \cite{Furtula}).

\begin{theorem} \label{t:forgotten}
We have for any graph $G$ with minimum degree $\d$, maximum degree $\D$ and $m$ edges
$$
2m-\frac{F(G)}{2\d^2}
\le GA_1(G)
\le  2m-\frac{F(G)}{2\D^2}
$$
and each inequality is attained if and only if $G$ is regular.
\end{theorem}

\begin{proof}
The equality
$$
\frac{2\sqrt{d_u d_v}}{d_u + d_v} \frac{2\sqrt{d_u d_v}}{d_u + d_v} +
\frac{2(d_u^2 + d_v^2)}{(d_u + d_v)^2} = 2
$$
and Corollary \ref{cor} give
$$
\begin{aligned}
\frac{2\sqrt{d_u d_v}}{d_u + d_v} + \frac{d_u^2 + d_v^2}{2\d^2} & \ge 2,
\\
GA_1(G) + \frac{F(G)}{2\d^2} & \ge 2m .
\end{aligned}
$$
We also have
$$
\begin{aligned}
\frac{2\sqrt{d_u d_v}}{d_u + d_v} + \frac{d_u^2 + d_v^2}{2\D^2} & \le 2,
\\
GA_1(G) + \frac{F(G)}{2\D^2} & \le 2m .
\end{aligned}
$$

If the graph is regular, then the lower and upper bounds are the same, and they are equal to $GA_1(G)$.
If a bound is attained, then we have either
$d_u+d_v = 2\d$ for every $uv\in E(G)$ or $d_u^2+d_v^2 = 2\D^2$ for every $uv\in E(G)$ and we conclude that $d_u = d_v$ for every $u,v\in V(G)$.
\end{proof}

\end{document}